\documentclass[12pt]{article}

\usepackage{amssymb}
\usepackage{amsmath}
\usepackage{amsthm}
\usepackage[T1]{fontenc}
\usepackage[utf8]{inputenc}
\usepackage[english]{babel}
\usepackage[section]{algorithm}
\usepackage{algorithmic}
\usepackage{graphicx} 
\usepackage[normalsize]{subfigure}   
\usepackage{caption}
\providecommand{\keywords}[1]
{
  \small	
  \textbf{\textit{Keywords : }} #1
}

\newtheorem{theorem}{Theorem}[section]
\newtheorem{definition}[theorem]{Definition}
\newtheorem{proposition}[theorem]{Proposition}
\newtheorem{lemma}[theorem]{Lemma}

\newtheorem{example}[theorem]{Example}

\DeclareMathOperator*{\Dom}{Dom\,}
\DeclareMathOperator*{\Bd}{Bd\,}
\DeclareMathOperator*{\Ima}{Im\,}
\DeclareMathOperator*{\Crit}{Crit\,}
\DeclareMathOperator*{\Cl}{Cl\,}

\DeclareMathOperator*{\card}{card\,}

\DeclareMathOperator*{\If}{If\,}

\usepackage{geometry}
\usepackage{float}
\renewcommand{\footnote}{\arabic{footnote}}

\begin{document}

	\title{From Finite Vector Field Data to Combinatorial Dynamical Systems in the Sense of Forman}
	\author{Dominic Desjardins Côté \footnote{Email : dominic.desjardins.cote@usherbrooke.ca}}
	
	\maketitle
	\bibliographystyle{plain}	
	\begin{abstract}
		The main goal is to construct a combinatorial dynamical system in the sense of Forman from finite vector field data. We use a linear minimization problem with binary variables and linear equality constraints. The solution of the minimization problem induces an admissible matching for the combinatorial dynamical system. They are three main steps for the method: Construct a simplicial complex, compute a vector for each simplex, solve the minimization problem, and apply the induced matching. We argue the effectiveness of the method by testing it on the Lotka-Volterra model and the Lorenz attractor model. We are able to retrieve the cyclic behaviour in the Lotka-Volterra model and the chaotic behaviour of the Lorenz attractor. Two extensions to the algorithm are shown. We use the barycentric subdivision to obtain a better resolution. We add conditions on the minimization problem to obtain a solution which induce a gradient matching. 
	\end{abstract}		
	
	
	\setlength{\parskip}{1em}
	\section{Introduction}
		
		The concept of combinatorial vector fields, introduced by Robin Forman \cite{RobForMorse} \cite{RobCVF} \cite{RobForUser}, is a useful tool to discretize continuous problems in mathematics \cite{thMorseClas}, in imaging \cite{arRobJohShe}, and in the computation of homology \cite{arDisMorTheAlgo}. We are interested in the combinatorial dynamical system, because it gives a qualitative approach to study dynamical systems. We are interested to study the global dynamics by approximating the underlying dynamical system with combinatorial vector field. More development in \cite{towFormalTie} shows that classical dynamical system and combinatorial dynamical system are similar on the level of Conley Index. In recent years, a generalization for combinatorial dynamical system is the combinatorial multivector field studied has been proposed by \cite{arMultiVector}. A reason for this generalization is that the original Forman's definition of combinatorial vector field is hard to implement. We show that it is possible to construct a combinatorial dynamical system in the sense of Forman with finite vector field data.

		There are algorithms to construct a multivector field \cite{arPerHomMD}, and combinatorial gradient vector field in the sense of Forman from simplicial complex with value on vertices in $ \mathbb{R} $ \cite{arRobJohShe} \cite{arKinKnuMra}, and $ \mathbb{R}^d $ \cite{multiDimDisMor}. But there is no algorithm to construct a Forman's combinatorial dynamical system with cyclic behaviour. There are multiple results that we can apply on a combinatorial dynamical system. As an example, we can construct a semi flow \cite{semiFlowMMTW}, a cell decomposition\cite{towFormalTie} and a Morse decomposition \cite{linkCombClass2} for combinatorial dynamical systems.   
		
		By means of an example, it has been argued in \cite{arMultiVector} that combinatorial dynamical system in the sense of Forman might not be best suited for applications. The article \cite{arMultiVector} generalize it to combinatorial multivector field. For more information on combinatorial vector field, we refer this article \cite{arMulVecConMorFor} to the reader. In this paper, we argue that we can obtain a combinatorial dynamical system by Forman from data and we find an approximation of the global behaviour of the dynamical system with our method.
		
	Let us take the same example at \cite{arMultiVector} in the section 3.4 and apply our method on it. We have the following equations:
	
	\begin{equation}
		\begin{cases}
			\frac{dx}{dt} = -y + x(x^2 + y^2 - 4)(x^2+y^2-1) \\
			\frac{dy}{dt} = x + y(x^2 + y^2 - 4)(x^2+y^2-1)
 		\end{cases}.
	\end{equation}	 \label{eqEx1} 		
		
		The dynamical system (\ref{eqEx1}) has a repulsive stationary point at $ (0,0) $. It also has an attracting periodic orbit with a radius $1$ with a center at $(0, 0)$ and a repulsive periodic orbit with a radius $2$ with a center at $(0, 0)$. For the dataset, we take the following data points $ X = \{ (0.22 + 0.44i, 0.22 + 0.44j) \mid i=-8, -7, -6, \ldots, 6, 7 \text{ and } j = -8, -7, -6, \ldots, 6, 7  \} $. We construct the cubical complex of this dataset. The length of the side of a square is $ 0.44 $. By applying our algorithm, we obtain the Figure (\ref{figIntroEx}). At the middle we obtain a critical square that represents a repulsive behaviour. The arrows in yellow are all part of an attracting periodic orbit. The arrows in purple are all  part of a repulsive periodic orbit. The blue arrows represent the gradient behaviour. We obtain the expected results of a repulsive stationary point, an attracting periodic orbit and a repulsive periodic orbit. 
		
		\begin{figure}
  	 		\center
  			\includegraphics[height=10cm, width=10cm, scale=1.00, angle=0 ]{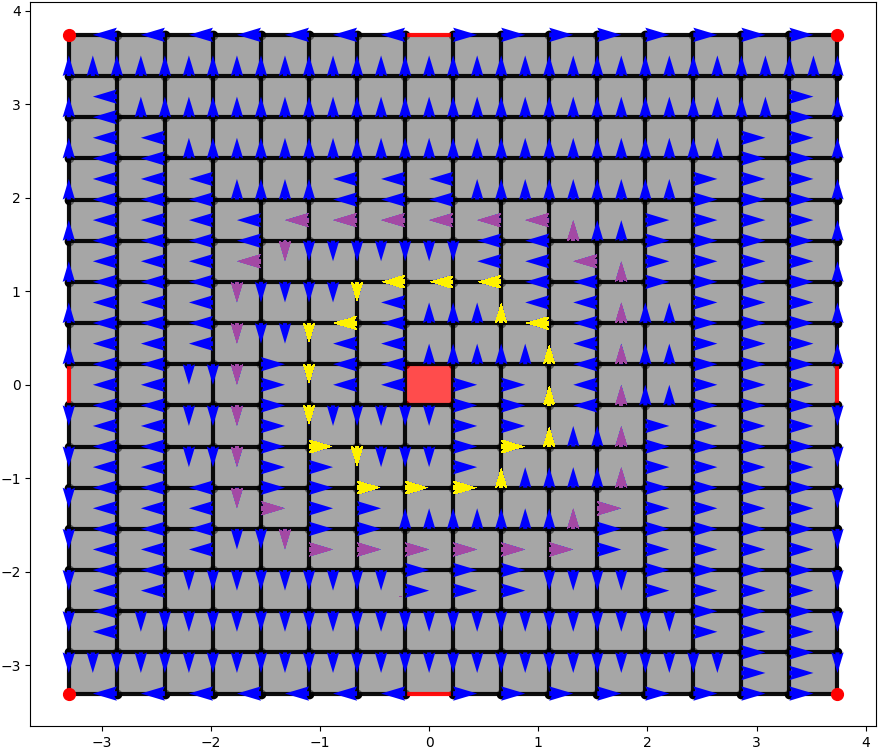}
  \caption{The combinatorial dynamical system obtained by applying our algorithm to the equations (\ref{eqEx1}) with the parameter $ \alpha = 0.90 $.}
  			\label{figIntroEx}
 		\end{figure} 
		
		The main result is the construction of a combinatorial dynamical system from finite vector field data. The data is $ X $ a set of points and $ \dot{X} $ the set of vector associates to each point. Let $ X = \{x_1, x_2, \ldots, x_N\}$ such as $ x_i \in \mathbb{R}^d $ and $ \dot{X} = \{ \dot{x_1}, \dot{x_2}, \ldots, \dot{x_N} \} $ such as each $ \dot{x_i} \in \mathbb{R}^d $ is a vector associate to the point $ x_i$. The algorithm is done in three steps. First, we need to construct a simplicial complex. For the purpose of this paper, we use a method based on the Delaunay complex and the Dowker complex. We can also use a cubical complex. Next, we assign a vector to each simplex in the simplicial complex. The assignment is different if we have a Delaunay complex or a Dowker complex. Finally, we match a $n$-simplex with a $(n+1)$-simplex or itself to satisfy the definition of a combinatorial dynamical system. We use a linear minimization with binary variables and linear equality constraints. Let $K$ a simplicial complex, $N$ the number of simplices in $K$ and $ x_i, x_j \in K$. Let's define $ A_N(\{0,1 \}) $ a matrix with binary entries. The matrix $ A $ is called the matching matrix and induce a matching. If $ a_{ij} = 1 $, then we match the simplex $ x_i $ to $ x_j $. If $ a_{ij} = 0$, then we don't match the simplex $ x_i $ to $ x_j $. For each $ a_{ij} $, we associate a cost $ c_{ij} $. We set a high cost if the matching is not admissible or the matching is bad. We put a low cost if the matching is good. If $ i = j $, then we set a fix cost to a parameter $ \alpha $. The constraints guarantee that each simplex is matched only once. We obtain the following minimization problem: 
		
	\begin{equation}\label{introEq}
		\begin{aligned}
		& \underset{A \in M_N(\{0,1\}) }{\text{minimize}}
		& & f(A) := \sum_{i=0}^{N-1}\sum_{j=0}^{N-1} a_{ij} c_{ij} \\
		& \text{subject to}
		& &  \sum_{i=0}^{N-1} a_{ik} + \sum_{j=0}^{N-1} a_{kj} - a_{kk}  = 1 , \quad k =0, 1, \ldots, N-1 
		\end{aligned}
	\end{equation}	
	    
	We have an integer linear problem(ILP).  We obtain our main result.
	
	\begin{theorem}
		If $ A$ is a global minimum to the ILP above, then $ A $ induce a matching that satisfies the definition of a combinatorial dynamical system. 
	\end{theorem}		
		
		The article is organized as follows. In Section 2, we explain basic definitions related to simplicial complex and combinatorial dynamical system. In Section 3, we show how to construct a simplicial complex by using the Delaunay Complex or the Dowker Complex. In Section 4, we assign a vector to each simplex. The assignment is different, if we use a Delaunay complex or a Dowker complex. In Section 5, we define the minimization problem. We also show that the global minimum of the minimization problem induces a matching that satisfies the definition of combinatorial dynamical system. In Section 6, we use our method to the Lotka-Volterra model and the Lorenz attractor model. In Section 7, we show two extensions to the algorithm. We can apply a barycentric subdivision to a simplicial complex before solving the optimization problem. The second extension is to obtain a  combinatorial dynamical system which is a gradient field. We can choose the parameter $\alpha$ so that the solution from the minimization problem is gradient or we can add constraints to the minimization problem removing cycles in the solution. 

	\section{Preliminaries}	
	
		In this section, we discuss about simplicial complex and combinatorial dynamical system in the sense of Forman. For more information on simplicial complex, we refer to the reader this book \cite{AlgTopo} for simplicial complex and these papers \cite{towFormalTie} \cite{RobCVF} for combinatorial dynamical systems.
			
		An abstract simplicial complex is a set $K$ that contains finite non-empty sets such that if $ A \in K $, then for all subsets of $ A $ is also in $K$. We also use the geometric simplex defined as follows. A geometric $ n $-simplex is the convex hull of a geometrically independent set $ \{ v_0, v_1, \ldots, v_n \} \subset \mathbb{R}^d $. This is the set of $ x $ such as $ x = \sum_{i=0}^{n} t_i v_i $ and $ \sum_{i=0}^n t_i = 1 $. $t_i$ are called barycentric coordinates. We denote $ [ v_0, v_1, \ldots, v_n ] $ a $n$-simplex spanned by the vertices $ v_0, v_1, \ldots, v_n $. A simplicial complex $K$ is a collection of simplices such that for all $ \sigma \in K $, if $ \tau \leq \sigma $, then $ \tau \in K $ and if $ \tau = \sigma_1 \cap \sigma_2 $ then $  \tau$ is either empty or a face of $ \sigma_1 $ and $ \sigma_2 $. We note $ K_n $ the set of all $n$-simplices in $K$. Any simplex spanned by the subsets of $ \{ v_0, v_1, \ldots, v_n  \} $ are called faces. We note $ \sigma \leq \tau $. If $ \sigma \leq \tau $ and $ \sigma \neq \tau $, then we say that $ \sigma $ is a proper face of $ \tau $. We note $ \sigma < \tau $. The closure of $ \sigma $ is the union of all faces of $ \sigma $ noted $ \Cl \sigma$. The boundary of $ \sigma $ is the union of all proper faces of $ \sigma $ noted $ \Bd \sigma $.
		
		We define a partial map $ f \subset X \times Y $ if $f$ is a relation and $ (x, y), (x, y') \in f $, then $ y = y' $. We write $ f: X \nrightarrow Y $. We define the domain and the image of $ f $ as follows:
		
		\begin{itemize}
			\item $ \Dom f := \{ x \in X \mid \exists y \in Y, (x, y) \in f \} $;
			\item $ \Ima f := \{ y \in Y \mid \exists x \in X, (x, y) \in f \} $.
		\end{itemize}
		
		A partial map $f$ is injective, if for $ x, x' \in X $ we have $ f(x) = f(x')  $ implies that $ x = x' $. We can now define a combinatorial dynamical system.	We use the same definition as \cite{towFormalTie}.	
		\begin{definition}\label{dfnSDC}
			Let $ K $ a simplicial complex. $ \mathcal{V} : K \nrightarrow K $ is a combinatorial dynamical system if $ \mathcal{V} $ is a partial injective map and satisfy these conditions:
			\begin{itemize}
				\item For all $ \sigma \in \Dom \mathcal{V} $, $ \mathcal{V}(\sigma) = \sigma $ or $ \mathcal{V}(\sigma) =  \tau $ such as $ \tau > \sigma $ and $ \dim \sigma + 1 = \dim \tau $.
				\item $ \Dom \mathcal{V} \cup \Ima \mathcal{V} = K $.
				\item $ \Dom \mathcal{V} \cap \Ima \mathcal{V} = \Crit \mathcal{V} $, where $ \Crit \mathcal{V} = \{ \sigma \in K \mid \mathcal{V}(\sigma) = \sigma \} $ is the set of critical simplices.
			\end{itemize}
		\end{definition}
		
		This definition is equivalent to the definition of combinatorial dynamical system in the sense of Forman \cite{RobCVF} \cite{RobForUser}. But our definition uses a partial map. The first condition defines the combinatorial vector of $ \mathcal{V}$. The second condition is that every simplex is in a matching. The third condition is that only critical simplices are in both the image and the domain of $ \mathcal{V} $.
		
	To visualize combinatorial dynamical system, we draw a vector from the barycenter of $ \sigma $ to the barycenter of $\mathcal{V}(\sigma) $ for each $ \sigma \in \Dom \mathcal{V} \setminus \Crit \mathcal{V} $. If $ \sigma \in \Crit \mathcal{V} $, then we color the critical simplex in red. In the Figure \ref{figSdcEx}, we have some examples of combinatorial dynamical system and some counter-examples at Figure \ref{figSdcCe}.
		
		\begin{definition}\label{dfnMatchAdmi}
			Let $ \sigma, \tau \in K $ with $ \sigma \neq \tau $. We say that a pair $(\sigma, \tau)$ is an admissible matching, if $ \sigma < \tau $ with $ \dim \sigma + 1 = \dim \tau $.
		\end{definition}
		
		This definition is useful in our minimization problem. 
		
		\begin{figure}
  			\center
  			\subfigure[ $ \mathcal{V} $ is not injective. ]{
   				\includegraphics[height=5.715cm, width=7.62cm, scale=1.00, angle=0 ]{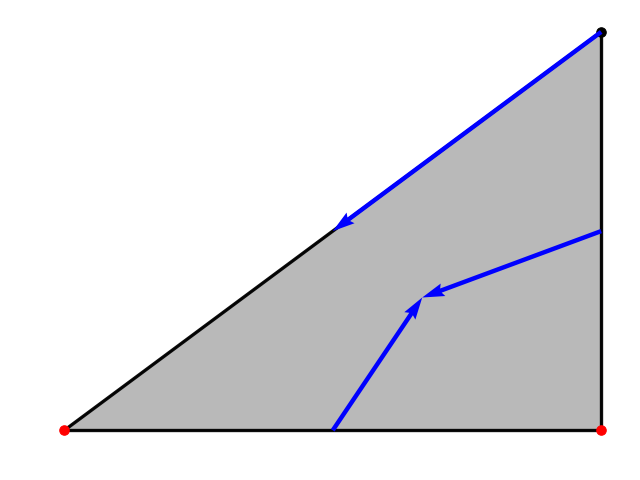}\label{subFigSdcCe1}
  			}
  			\,
  			\subfigure[ The third condition is not satisfied. ]{
   				\includegraphics[height=5.715cm, width=7.62cm, scale=1.00, angle=0 ]{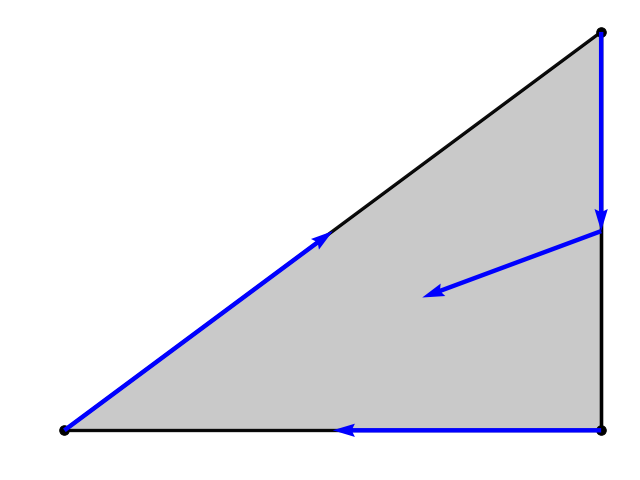}\label{subFigSdcCe2}
  }
  			\qquad \qquad \qquad
  			\subfigure[ The first condition is not satisfied. ]{
   			\includegraphics[height=5.715cm, width=7.62cm, scale=1.00, angle=0 ]{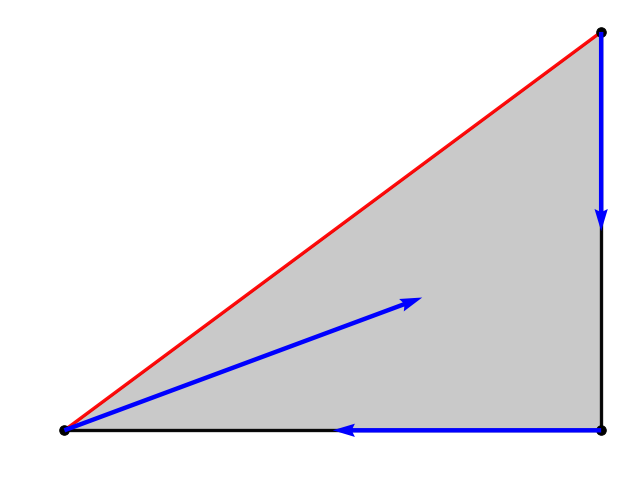}\label{subFigSdcCe3}
  }
  			\, 
    		\subfigure[ The second condition is not satisfied. ]{
    		\includegraphics[height=5.715cm, width=7.62cm, scale=1.00, angle=0 ]{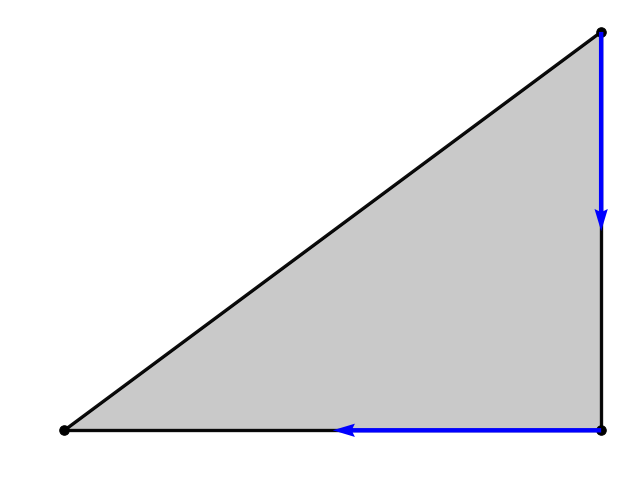}
   			 \label{subFigSdcCe4}
 }
   			 
   			 \qquad 
    		\subfigure[ $ \mathcal{V} $ is not a partial map. ]{
    		\includegraphics[height=5.715cm, width=7.62cm, scale=1.00, angle=0 ]{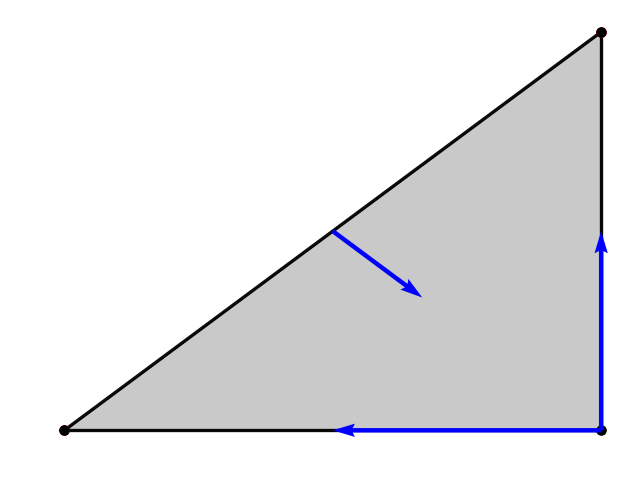}
   			 \label{subFigSdcCe5}
  }
  			\caption{Counter-examples of combinatorial dynamical systems.}
  			\label{figSdcCe}
 		\end{figure}
 		
 		\begin{figure}
  			\center
  			\subfigure[ A combinatorial dynamical system in $ \mathbb{R}^2 $.]{
   			\includegraphics[height=5.715cm, width=7.62cm, scale=1.00, angle=0 ]{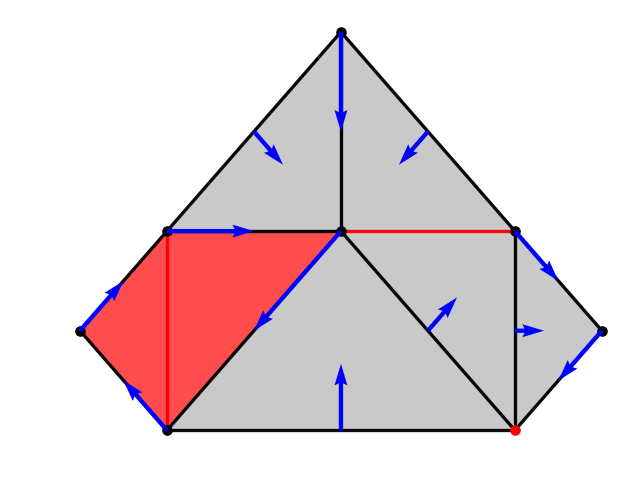}\label{subFigSdcEx1}
  		}
  		\,
  		\subfigure[ A combinatorial version of the Lorenz attractor in $ \mathbb{R}^3 $. ]{
   			\includegraphics[height=5.715cm, width=7.62cm, scale=1.00, angle=0 ]{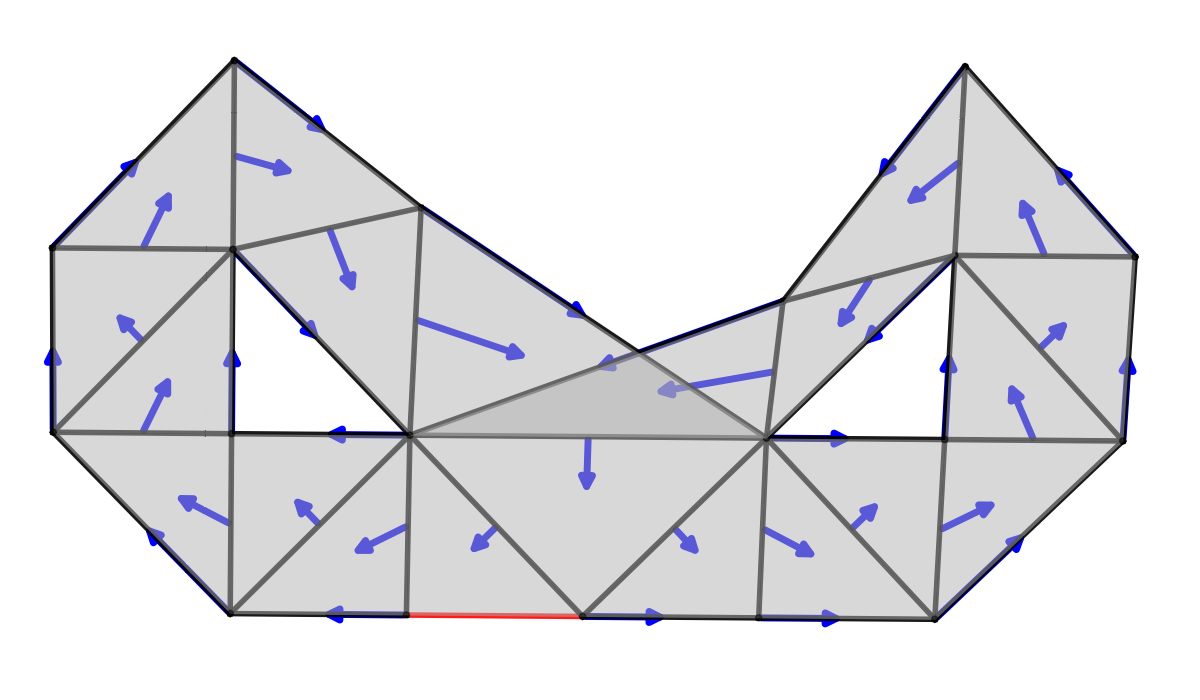}\label{subFigSdcEx2}
  }
    		\caption{Two examples of combinatorial dynamical system.}
  			\label{figSdcEx}
 		\end{figure}		
		
		A multivalued map is a map $ F: X \to P(X) $ where $P(X)$ is the power set of $X$. We write a multivalued map by $ F : X \multimap X $.
		
		\begin{definition}
			A combinatorial multi-flow associated with a combinatorial dynamical system $ \mathcal{V} $ is $ \Pi_{\mathcal{V}} : K \multimap K $ :
			\begin{equation}
				\Pi_{\mathcal{V}}(\sigma):= 
				\begin{cases}
					\Cl \sigma & \If \, \sigma \in \Crit \mathcal{V} \\
					\Bd \sigma \setminus \{ \mathcal{V}^{-1}(\sigma) \} & \If \, \sigma \in \Ima \mathcal{V} \setminus \Crit \mathcal{V} \\
					\{ \mathcal{V}(\sigma) \} & \If \, \sigma \in \Dom \mathcal{V} \setminus \Crit \mathcal{V}. \\
				\end{cases}
			\end{equation}						
			
		\end{definition}
		
				$ \Pi_{\mathcal{V}} $ induces the dynamics in combinatorial dynamical systems. 
		
			\begin{definition}
				A solution of a combinatorial multi-flow $ \Pi_{\mathcal{V}} $ of a combinatorial dynamical system $ \mathcal{V} $ is a function $ \varrho : I \to K $ such as $ I $ is an interval in $ \mathbb{Z} $ and $ \varrho(i + 1) \in \Pi_{\mathcal{V}}(\varrho(i)) $, for all $ i \in I $. We said it's a full solution if $ I = \mathbb{Z} $. 
			\end{definition}
		
	We define a cycle to be a solution $  \varrho$ with $ I = [0, n] \cap \mathbb{Z} $ such as $ \varrho(0) =  \varrho(n)$. We say a cycle is elementary, if every simplex is unique in $ \varrho(i) $ for $ i = 1,2, \ldots, n-1 $. The image of a cycle has $n$-simplex and $(n+1)$-simplex or a critical simplex. Because the only way to go up in dimension is from a simplex in $ \Dom \mathcal{V} \setminus \Crit \mathcal{V} $ and we cannot have two simplices $ \in \Dom \mathcal{V} \setminus \Crit \mathcal{V} $ in a row inside a solution. We denote a $d$-cycle where $d$ is the dimension of simplex in $ \Dom \mathcal{V} \cap \Ima \varrho $ and $ \Crit \mathcal{V} \cap\Ima \varrho = \emptyset $. We can use the strongly connected components from the directed graph of $ \Pi_{\mathcal{V}} $ to understand the recurrent dynamics of $ \mathcal{V} $. This will help us to study the method in the experiments. We say that a cycle $\varrho$ self-intersect at $ \tau $, if $ \card(\Ima \varrho \cap \Pi_{\mathcal{V}}(\tau)) > 1$.
			
	\section{Construction of a Simplicial Complex}
		In this section, we want to construct a simplicial complex. We remind the dataset is $ X \in \mathbb{R}^d $. For each $x \in X$, there is an associate vector $ \dot{x} \in \mathbb{R}^n $. In this article, we will show two different ways to construct a simplicial complex : the Delaunay complex \cite{boCompTopo} and the Dowker complex \cite{rGhristEAT}.
			
			\begin{definition}
				The Voronoi cell of a point $u \in S$ is the set of points, $V_u = \{ x \in \mathbb{R}^d \mid \| x- u \| \leq \| x - v \|, v \in S \}$. The Voronoi diagram of $S$ is the collection of Voronoi cells of its points.
			\end{definition}
			
			\begin{definition}
				The Delaunay complex $K$ of a finite set $ S \subset \mathbb{R}^d $ is isomorphic to the nerve of the Voronoi diagram,
				
				\begin{equation}
					K = \{ \sigma \subseteq S \mid \cap_{u \in \sigma} V_u \neq \emptyset \}.
				\end{equation}								
				
			\end{definition}
			
			The Delaunay complex works better if the sampling of the data is uniform. When the sampling is not uniform, this causes a great variation in the volume of $n$-simplices. It is harder to analyze the data. The Dowker complex is a solution for a non-uniform sampling of finite vector field data.		
	
			\begin{definition}
				Let $X$ and $Y$ sets and let $R \subset Y \times X$ be a relation. The Dowker complex $K$ is given by:
				\begin{equation*}
					K := \{[y_0, y_1, y_2, \ldots, y_n] \mid \text{ there exists a } x \in X \text{ such that } yRx_i \text{ for all } i= 0, 1, \ldots, n \}.
				\end{equation*}
			\end{definition} 
			
			For the purpose of this paper, we use $X$ to store data and we need to choose the set $ Y $ and a relation $R$ between $X$ and $Y$. 
		
			\begin{example}\label{exDowCmp}
				Let $X = \{ (-1.2, 0), (0, 0.5), (0, -0.5), (0.5, 0)) \} \subset \mathbb{R}^2$ and $Y = \{ y_0, y_1, y_2, y_3, y_4\}$ where $ y_0 = (0, 0), y_1 = (-1, -1), y_2 = (1, 1), y_3 =(-1, 1), y_4 = (1, -1) $. $x \in X $ and $y \in Y $ are in relation if $ \| y - x \|_2 < 1.2$. We obtain the Dowker complex is $ \{ [y_0, y_1, y_4], [y_0, y_2, y_4], [y_0, y_2, y_3], [y_1, y_3] \} $.
				
				\begin{figure}
  	 				\center
  					\includegraphics[height=3.75in, width=5in, scale=1.00, angle=0 ]{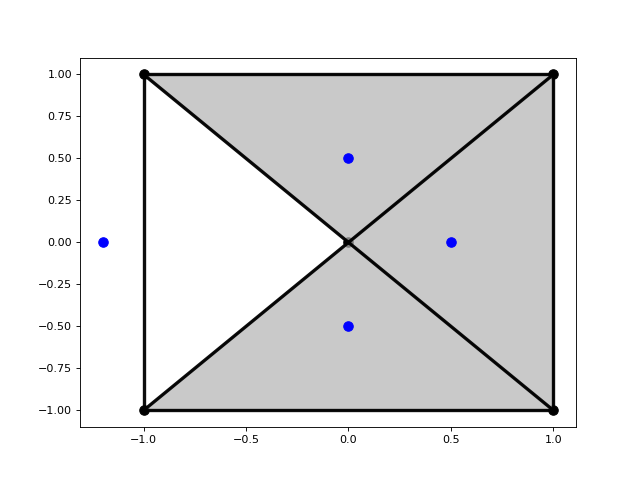}
  \caption{The Dowker complex obtained from the example (\ref{exDowCmp}).}
  					\label{imgDowkerComplex}
 				\end{figure}
			\end{example}	
		
	\section{Compute a Vector for each Simplex}
	
		In this section, we define two ways of assigning a vector to each simplex. We construct an application $ V : K \to \mathbb{R}^n $.
		
		Let $K$ be a Delaunay complex. We compute a vector for a simplex by taking an average of vectors from the vertices of the simplex. More precisely, let $ \sigma = [ x_0, x_1, \ldots, x_n ] $:
		
		\begin{equation}
			V(\sigma) := \frac{\sum_{i=0}^n \dot{x}_i}{n+1}
		\end{equation}		 
		
		Let $K$ be a Dowker complex. Let $ X $ and $\dot{X}$ from the data and $ Y $ a set and a relation $ R $ between $X$ and $Y$. We define an application $ w : K \to X $ by $w(\sigma) = \lbrace \dot{x} \in \dot{X} \mid \text{if for all } y \in \sigma, x R y \rbrace$.
		
		We assign a vector to a simplex $ \sigma = [x_0, x_1, \ldots, x_n ] $:
	
		\begin{equation}
			V(\sigma) := \frac{\sum_{\dot{x}\in w(\sigma)} \dot{x}  }{ card(w(\sigma))}
		\end{equation}
	
	\section{The Matching Algorithm}
	
		In this section, we describe the matching algorithm. We minimize a linear program with binary variables and linear equality constraints. 
		
		First, let's define the variables. Let $N$ the number of simplices and $ x_i $ the simplex where $0 \leq i < N$. We define a matrix $ A \in M_{N}(\{ 0,1 \}) $ of dimensions $N \times N$ where the entries are binary values. We call $ A $ the matching matrix. We do the matching as follows :
		
		\begin{itemize}
			\item If $ a_{ij} = 1 $, match $ x_i $ to $x_j$ $(\mathcal{V}(x_i) = x_j)$
			\item If $ a_{ij} = 0 $, do not match $ x_i $ to $x_j $ $ (\mathcal{V}(x_i) \neq x_j) $
 		\end{itemize}			
 		
 		Now, let us define the matrix $ C \in M_{N}(\mathbb{R}) $ called the cost matrix. Let $ V : K \to \mathbb{R}^n $ be the map that takes a simplex and returns the vector value from the data. Let $b : K \to \mathbb{R}^n $ be the map sending each simplex to its barycenter and $ W : K \times K \to \mathbb{R}^n $ be given by $ W(x_i, x_j) = b(x_j) - b(x_i) $. Then, $ W(x_i, x_j) $ is a vector that starts from the barycenter of $ x_i $ and ends at the barycenter of $x_j$. The main idea is to compare the vector $ V(x_i) $ to the vector $ W(x_i, x_j) $ when the matching is admissible. Given $ x_i \neq x_j $, the matching is admissible if $ x_i < x_j $ and $ \dim x_i + 1 = \dim x_j $. We use the cosine distance to compare them. It is defined as follows:
 		
 		\begin{equation}
 			d(x, y) = 1 - \cos(\theta) = 1 - \frac{x \cdot y}{ \| x\| \| y \|},
		\end{equation}
	where $ \theta $ is the angle between $ x $ and $ y$. If $ d(x,y) = 0 $, then $ x $ and $y$ are parallel in the same direction. If $ d(x,y) = 1 $, then $ x $ and $y$ are perpendicular. If $ d(x,y) = 2 $, then $ x $ and $y$ are parallel in opposite directions. We set $ c_{ij} = d(V(x_i), W(x_i, x_j)) $ when $ x_i $ and $ x_j $ is an admissible matching and $ V(x_i) \neq 0 $.
	
	If $ V(x_i) = 0 $ and $ a_{ij} $ is an admissible matching, then $ c_{ij} = 2$. Indeed, $ d(x, 0) = 1 $ for every $ x \in \mathbb{R}^d $. This means that all vectors are perpendicular to $ \vec{0} $. But $x_i$ should be a critical simplex. By setting higher value for his cost, it has a higher chance that the solution will set this simplex to be critical.
		
	If $ i = j $, then we set to value $ c_{ii} $ to $ \alpha \in [0, 2] $. $ \alpha $ is a parameter choose by the user. If the $ \alpha $ value is low, then we have more critical simplices. If the $ \alpha $ value is high, then we have less critical simplices. 
	
	Let us explain the geometric interpretation of the parameter $ \alpha $. There exists $ \beta \in [-1, 1] $ such as $ \alpha = 1 - \beta $. Fix a simplex $ x_i $ and consider all the admissible matching $ a_{ij} $. Suppose that for all $ c_{ij} > \alpha $ and let $ \theta_j $ the angle between $ V(x_i) $ and $ W(x_i, x_j) $. We obtain:
	
	\begin{align*}
		& 1 - \beta <1 - d(V(x_i), W(x_i, x_j))  \\ 
		& \implies 1-\beta <1 - \cos(\theta_j)  \\
		& \implies \beta > \cos(\theta_j)\\
		& \implies \arccos(\beta) < \theta_j,
	\end{align*}	  
	where $ \arccos(\beta) $ represents the critical angle. If for all $ \theta_j  > \arccos(\beta) $, then the minimization problem (\ref{eqPrOp}) put simplex $ x_i $ is a critical simplex or it matches with another simplex of lower dimensions.
	
	If the matching is not admissible and $ i \neq j $, then $ C_{ij} = \max(2\alpha + 1, 3) $.
	
	Finally, we obtain this equality for the entries of the matrix $ C $ : 
	
	\begin{equation}
		C_{ij} = \begin{cases}
			d(V(x_i), W(x_i,x_j))	& \text{if } (x_i, x_j) \text{ is admissible and } V(x_i) \neq \vec{0} \\
			2 & \text{if } (x_i, x_j) \text{ is admissible and } V(x_i) = \vec{0} \\
			\alpha & \text{if } i = j \\
			\max(2\alpha + 1, 3)  & \text{Otherwise}
		\end{cases}.
	\end{equation}

	\begin{figure}
  	 	\center
  		\includegraphics[height=3in, width=4in, scale=1.00, angle=0 ]{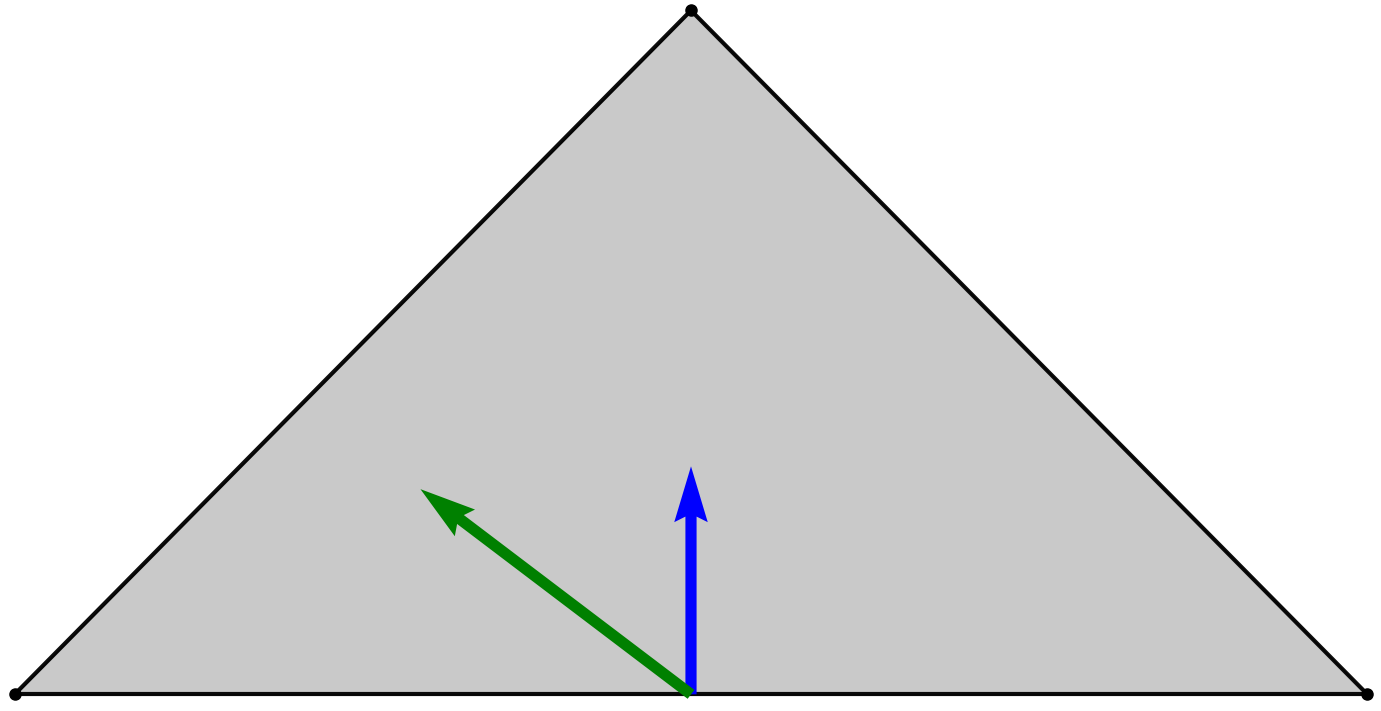}
  \caption{We compare $V([v_0, v_1])$ in green and  $ W([v_0, v_1], [v_0, v_1, v_2]) $ in blue with the cosine distance.}
  \label{compVW}
 \end{figure}
		
		Finally, we obtain the following minimization linear problem with binary variables and linear equality constraints:
		
	\begin{equation}\label{eqPrOp}
		\begin{aligned}
		& \underset{A \in M_N(\{0,1\}) }{\text{minimize}}
		& & f(A) := \sum_{i=0}^{N-1}\sum_{j=0}^{N-1} a_{ij} c_{ij} \\
		& \text{subject to}
		& &  \sum_{i=0}^{N-1} a_{ik} + \sum_{j=0}^{N-1} a_{kj} - a_{kk}  = 1 , \quad k =0, 1, \ldots, N-1 
		\end{aligned}
	\end{equation}	
	
	Let us explain the double sum in the constraint for a fixed $k$. The sum $ \sum_{i=0}^{N-1} a_{ik} $ counts the number of vectors going into $ x_i $ i.e. $ \mathcal{V}(x_i) = x_k  $. The $ \sum_{j=0}^{N-1} a_{kj} $ counts the number of vectors going out of $ x_i $ i.e $ \mathcal{V}(x_k) = x_i $. We remove the term $ a_{kk} $ because it is counted twice. 
	
	\begin{theorem}\label{thGloMinOpt}
		If $ A$ is a global minimum of the ILP (\ref{eqPrOp}), then $ A $ induce a matching that satisfies the definition of a combinatorial dynamical system. 
	\end{theorem}
	
	\begin{proof}
		Let us show that there exists a global minimum for the optimization problem (\ref{eqPrOp}) and it induces an admissible matching.  We set $ a_{kk} = 1 $ for all $ k =0, 1, \ldots, N-1 $ and $ a_{ij} = 0 $ for $ i \neq j $. $A$ satisfies the constraints of (\ref{eqPrOp}) and $A$ induces the matching that all simplices are critical. Then, the solution set of (\ref{eqPrOp}) is always feasible. The variables $ x_i $ are bounded between $0$ and $1$ for all $i = 0, 1, \ldots, N-1$. Because the set of solutions is bounded and feasible, there exists a global minimum for (\ref{eqPrOp}).
	
		Let's show by contraposition that if the matrix $ A$ does not induce an admissible matching, then it does not satisfy the constraint of the minimization problem (\ref{eqPrOp}). We have $5$ cases that can be seen in Figure \ref{figSdcCe}. If there are two vectors going out of $ x_k$, then there exist $i,j $ such as $ a_{ki} = 1 $ and $ a_{kj} = 1 $ for $ i \neq j $. That implies that the constraints of (\ref{eqPrOp}) are not satisfied. If there are two vectors going in $ x_k $ then there exist $i,j $ such as $ a_{ik} = 1$ and $ a_{jk} = 1 $ for $ i \neq j $. The constraints of (\ref{eqPrOp}) are not satisfied. If there is a vector going in $ x_k $ and another vector going out of $ x_k $, then there exist $i,j $ such as $ a_{ik} = 1 $ and $ a_{kj} = 1 $. Then, the constraints of (\ref{eqPrOp}) are not satisfied. We obtain that $ \Dom \mathcal{V} \cap \Ima \mathcal{V} = \Crit \mathcal{V} $. For all $ x_k \in K $, there exists $i$ such as $ a_{ik} = 1 $ or $ a_{ki} = 1 $. This means that either $ x_k \in \Dom \mathcal{V} $ or $ x_k \in \Ima \mathcal{V} $. We obtain the equality $\Dom \mathcal{V} \cup \Ima \mathcal{V} = K $. 		
		
		We need to check that if $ A $ have an inadmissible matching then $A$ is not a global minimum. Let's show it by contradiction. Let $ A$ by the global minimum of the problem (\ref{eqPrOp}) and $ a_{ij} = 1 $ where ($ x_i $, $ x_j $) is an inadmissible matching with $ i \neq j $. Since $ a_{ij} = 1 $, $ c_{ij} = \max(2\alpha + 1, 3) $. But, if $ a_{ii} = 1 $ and $ a_{jj} = 1$, then $ c_{ii} + c_{jj} = 2\alpha < 2\alpha +1 = c_{ij} $. So, if we change the variables for $ a_{ii} = 1, a_{jj} = 1$ and $ a_{ij} = 0 $, it still satisfies the constraints, and we obtain a smaller value for the objective function of (\ref{eqPrOp}). Then, $A$ is not the global  minimum of (\ref{eqPrOp}), and we obtain the contradiction. 		
	\end{proof}
		
	 If $A$ is not an admissible matching for a combinatorial dynamical system, then we can always find a matrix $B$ that satisfies the constraint of (\ref{eqPrOp}), and $ f(B) < f(A) $.
	
	\begin{proposition}\label{propNotAdmi}
		If $ A $ is not an admissible matching for a combinatorial dynamical system, and it satisfies the constraints of the minimization problem (\ref{eqPrOp}), then we can find $B$ such that $B$ is an admissible matching, and $ f(B) < f(A) $.
	\end{proposition}
	
	\begin{proof}
		If $ A $ is not an admissible matching and it satisfies the constraint of (\ref{eqPrOp}), there exists some $ (i,j) $ such as $ a_{ij} = 1 $ and ($x_i,x_j) $ is not an admissible matching for the definition of the combinatorial dynamical system. 
		
	Let us construct the matrix $B \in M_{N}(\{0, 1 \})$. If $ (x_i, x_j) $ is an admissible matching, then $ a_{ij} = b_{ij} $. If $ x_i \to x_j $ is not an admissible matching, then $ b_{ii} = b_{jj} = 1 $ and $ b_{ij} = 0 $. We have $ c_{ii} = c_{jj} = \alpha $ and $ c_{ij} = \max(2\alpha + 1, 3) $. Then, $ c_{ii} + c_{jj} = 2 \alpha < c_{ij} $.
	
	Finally, we obtain that $B$ still satisfies the constraints of (\ref{eqPrOp}), and $ f(B) < f(A) $.   
	\end{proof}
	
	We can use the proposition (\ref{propNotAdmi}) to transform $A$ to an admissible matching with a simple procedure by setting simplices in inadmissible matching too critical.
	
	Let us explain the meaning of the value of $f(A)$ of the problem (\ref{eqPrOp}). Let $A$ be a solution of (\ref{eqPrOp}) and suppose that $ V(x_i) \neq \vec{0} $ for all $i$.	Let $ \mathcal{I} $ be the set of $(i, j)$ such that $ a_{ij} = 1 $ with $ i \neq j $. Let $ \mathcal{K} $ be the set of $ k $ such that $ a_{kk} = 1 $. 
	
	\begin{align*}\label{eqMinValueF}
		f(A) :&= \sum_{i=0}^{N-1} \sum_{j=0}^{N-1} c_{ij}a_{ij} = \sum_{(i,j) \in \mathcal{I}} c_{ij} + \sum_{k \in \mathcal{K}} c_{kk} \\
			  &= \sum_{(i, j) \in \mathcal{I}} \left(1 - \frac{V(x_i) \cdot W(x_i, x_j)}{\| V(x_i) \| \| W(x_i, x_j) \|} \right) + \sum_{k \in \mathcal{K}} \alpha \\
			  &= | \mathcal{I} | - \sum_{(i,j) \in \mathcal{I}} \frac{V(x_i) \cdot W(x_i, x_j)}{\| V(x_i)\| \| W(x_i, x_j) \| }  + | \mathcal{K} | \alpha  
	\end{align*}
	where $ | \mathcal{I} |$ is the number of matchings, $ | \mathcal{K} |$ is the number of critical simplices and the sum in the middle is the sum of the cosine angle between $ V(x_i) $ and $W(x_i, x_j) $. If $ \alpha $ is high, then the minimization will set more admissible matching. If $ \alpha $ is low, then the minimization will set more critical simplices.
			
	If our problem has a lot of simplices, we could only consider the variables $ a_{ij} $ such as $(x_i , x_j)$ is an admissible matching. This reduces greatly the number of variables for the minimization problem (\ref{eqPrOp}). Let's define the minimization problem with reduced binary variables.
	
	Let $ z $ be a vector such that $ z_k $ is assigned to an admissible matching $( x_i, x_j )$ and $ z_k \in \{0, 1 \} $. Let $m$ be the number of variables in $z$. Let $ c $ be the vector with $ c_k = c_{ij} $ for the variable  $ z_k $. Let $D \in M_{N, m}(\{0,1\})$ the constraint matrix. We define the entries of $D$ as follows. Let $ x_i $ be a simplex and $ z_j $ be assigned to the matching $ x_a \to x_b$. Then $ D_{i,j} = 1 $, if $ i = a $ or $ i=b$. Otherwise $ D_{i,j} = 0 $. We obtain the new minimization problem :
		 
	\begin{equation}\label{eqPrRedOpt}
		\begin{aligned}
		& \underset{z_k \in \{0,1\} }{\text{minimize}}
		& & f(z) = c \cdot z \\
		& \text{sujet à}
		& & Dz  = \vec{\mathbf{1}}_m
		\end{aligned}\qquad .
	\end{equation}		
	We put a lower bound and an upper bound to the number of variables of the minimization problem (\ref{eqPrRedOpt}). We can now easily show the next Lemma \ref{lemNbVar}.
	
	\begin{lemma}\label{lemNbVar}
	 Lets $m$ the number of variables and $N$ the number of simplices. Then, $ N \leq m \leq N^2$. 
	 \end{lemma}
	
	\section{Examples}
		
		\subsection{A Toy Example}
			We discuss here in detail a simple example aimed at understanding how the algorithm works.
		
		Let us define the data. We have $ v_0 =(0,0) $, $ v_1 =(1,1), $ and $ v_2 = (2,0) $ with their associate vectors $ \dot{v_0} = (0,1) $, $ \dot{v_1} = (1,0), $ and $ \dot{v_2} = (-1, -1) $. 
			
		We construct the Delaunay complex from the points $ v_0, v_1 $ and $ v_2 $ and obtain the simplicial complex $ K = \{ [v_0, v_1, v_2], [v_0, v_1], [v_1, v_2], [v_0, v_2], [v_0], [v_1], [v_2] \} $. 
		
		Let us compute the application $ V: K \to \mathbb{R}^2 $. For the $0$-simplices, we have $ V([v_0]) = \dot{v_0} = (0, 1), V([v_1]) = \dot{v_1} = (1, 0)  $ and $ V([v_2]) = \dot{v_2} = (-1, -1) $. For the other simplices, we take the average of the vectors on the vertices of the simplex : 
	\begin{align*}
		V([v_0, v_1]) = \frac{V([v_0]) + V([v_1])}{2} = \frac{(0,1) + (1,0)}{2} =(\frac{1}{2}, \frac{1}{2})\\
		V([v_0, v_2]) = \frac{V([v_0]) + V([v_2])}{2} = \frac{(0,1) + (-1, -1)}{2} =(-\frac{1}{2}, 0) \\
		V([v_1, v_2]) = \frac{V([v_1]) + V([v_2]) }{2} = \frac{(1,0) + (-1,-1)}{2} =(0, - \frac{1}{2}) \\
		V([v_0, v_1, v_2]) = \frac{V([v_0]) + V([v_1]) + V([v_2])}{3} = \frac{(0,1)+(1,0)+(-1,-1)}{3} =(0,0)
	\end{align*}
	
	We establish the index for the simplices. 
	\begin{equation*}
		x_0 = [v_0], x_1 = [v_1], x_2 = [v_2], x_3 = [v_0, v_1], x_4 = [v_0, v_2], x_5 = [v_1, v_2] \text{ and } x_6 = [v_0, v_1, v_2] 
	\end{equation*}
	
	We construct the matrix $C$  with $ \alpha = 0.75$. We have $ c_{ii} = \alpha = 0.75 $ for all $ i$. For the entries $c_{ij}$, where $ (x_i, x_j) $ is not an admissible matching, we have:
		\begin{align*}
			c_{01} = c_{02} = c_{05} = c_{06} = c_{10} = c_{12} = c_{14} = c_{16} = c_{20} = \\
			 c_{21} = c_{23} = c_{26} = c_{30} = c_{31} = c_{32} = c_{34} = c_{35} =  c_{40} =\\
			  c_{41} = c_{42} = c_{43} = c_{45} = c_{50} = c_{51} = c_{52} = c_{53} = c_{54} = \\
			c_{60} = c_{61} = c_{62} = c_{63} = c_{64} = c_{65} = \max(2\alpha+1, 3) = 3.
		\end{align*}
	
	Let's compute $ c_{ij}$ where $ (x_i, x_j)$ is an admissible matching. The set of admissible matching is $ \{ a_{03}, a_{04}, a_{13}, a_{15}, a_{24}, a_{25}, a_{36}, a_{46}, a_{56} \} $ where $ i \neq j $. Let's compute $ c_{03} $ and $c_{46} $ as an example. For $ c_{03} $, we need to compute $ W([v_0],[v_0, v_1]) $:

	\begin{align*}
		W([v_0], [v_0, v_1]) = b([v_0, v_1]) - b([v_0]) = (\frac{1}{2}, \frac{1}{2}) - (0,0) = (\frac{1}{2}, \frac{1}{2})
	\end{align*}
	Then,
	\begin{align*}
		c_{03} = d(V([v_0]), W([v_0, v_1])) = 1 - \frac{V([v_0]) \cdot W([v_0], [v_0, v_1])}{\| V([v_0]) \| \cdot \| W([v_0],[v_0, v_1]) \|} \\
		= 1 - \frac{(0,1)\cdot(\frac{1}{2}, \frac{1}{2})}{\| (0,1)\| \cdot \| (\frac{1}{2}, \frac{1}{2}) \| } = 1 - \frac{\sqrt(2)}{2} \approx 0.29
	\end{align*}			
		
	For $ c_{46} $, we need to compute $ W([v_0, v_2],[v_0, v_1, v_2]) $.
	\begin{align*}
		W([v_0, v_2], [v_0, v_1, v_2]) = b([v_0, v_1, v_2]) - b([v_0, v_2]) \\
		= \frac{(0,0) + (1,1) + (2,0)}{3} - \frac{(0,0) + (2,0)}{2}  = (0, \frac{1}{3})
	\end{align*}
	Then,
	\begin{align*}
		c_{46} = d(V([v_0, v_2]), W([v_0, v_2], [v_0, v_1, v_2])) \\ 
		= 1 - \frac{V([v_0, v_2]) \cdot W([v_0, v_2], [v_0, v_1, v_2])}{\| V([v_0, v_2]) \| \cdot \| W([v_0, v_2], [v_0, v_1, v_2]) \|} \\
		= 1 - \frac{(-\frac{1}{2}, 0) \cdot (0, \frac{1}{3})}{\|(-\frac{1}{2}, 0)\| \cdot \| (0, \frac{1}{3}) \|} = 1.00
	\end{align*}
		
	Finally, we obtain the matrix $ C $. We round it up to the last two decimals.
	
	\begin{equation}
		C = \begin{bmatrix}
		0.75 & 3 	& 3 	& 0.29 	& 1 	& 3 	& 3 \\
		3 	 & 0.75 & 3 	& 1.71 	& 3 	& 0.29 	& 3 \\
		3 	 & 3 	& 0.75 	& 3 	& 0.29 	& 1 	& 3 \\
		3 	 & 3 	& 3 	& 0.75 	& 3 	& 3 	& 0.55 \\
		3 	 & 3 	& 3 	& 3 	& 0.75 	& 3 	& 1 \\
		3 	 & 3 	& 3 	& 3 	& 3 	& 0.75 	& 0.86 \\
		3 	 & 3 	& 3 	& 3 	& 3		& 3 	& 0.75		
		\end{bmatrix}
	\end{equation}
	
	Now, we have everything set up to solve the minimization problem (\ref{eqPrOp}). We use our favourite solver to obtain this matching matrix $A$:
	
	\begin{equation}\label{eqMatToyEx}
		A = \begin{bmatrix}
		0 	 & 0 	& 0 	& 1 	& 0 	& 0 	& 0 \\
		0 	 & 0	& 0 	& 0 	& 0 	& 1 	& 0 \\
		0 	 & 0 	& 0 	& 0 	& 1 	& 0 	& 0 \\
		0 	 & 0 	& 0 	& 0 	& 0 	& 0 	& 0 \\
		0 	 & 0 	& 0 	& 0 	& 0 	& 0 	& 0 \\
		0 	 & 0 	& 0 	& 0 	& 0 	& 0 	& 0 \\
		0 	 & 0 	& 0 	& 0 	& 0		& 0 	& 1		
		\end{bmatrix}
	\end{equation}
	
	The matrix $A$ induces the matching: 
	\begin{align*}
		a_{03} &= 1 \implies \mathcal{V}(x_0) = x_3 \implies \mathcal{V}([v_0]) = [v_0, v_1] \\
		a_{15} &= 1 \implies \mathcal{V}(x_1) = x_5 \implies \mathcal{V}([v_1)] = [v_1, v_2] \\
		a_{24} &= 1 \implies \mathcal{V}(x_2) = x_4 \implies \mathcal{V}([v_2]) = [v_0, v_2] \\
		a_{66} &= 1 \implies \mathcal{V}(x_6) = x_6 \implies \mathcal{V}([v_0, v_1, v_2]) = [v_0, v_1, v_2].
	\end{align*}			
	
	We see that $ [ v_0, v_1, v_2 ] $ is a critical simplex.	
	
	\begin{figure}
  	 	\center
  		\includegraphics[height=3in, width=4in, scale=1.00, angle=0 ]{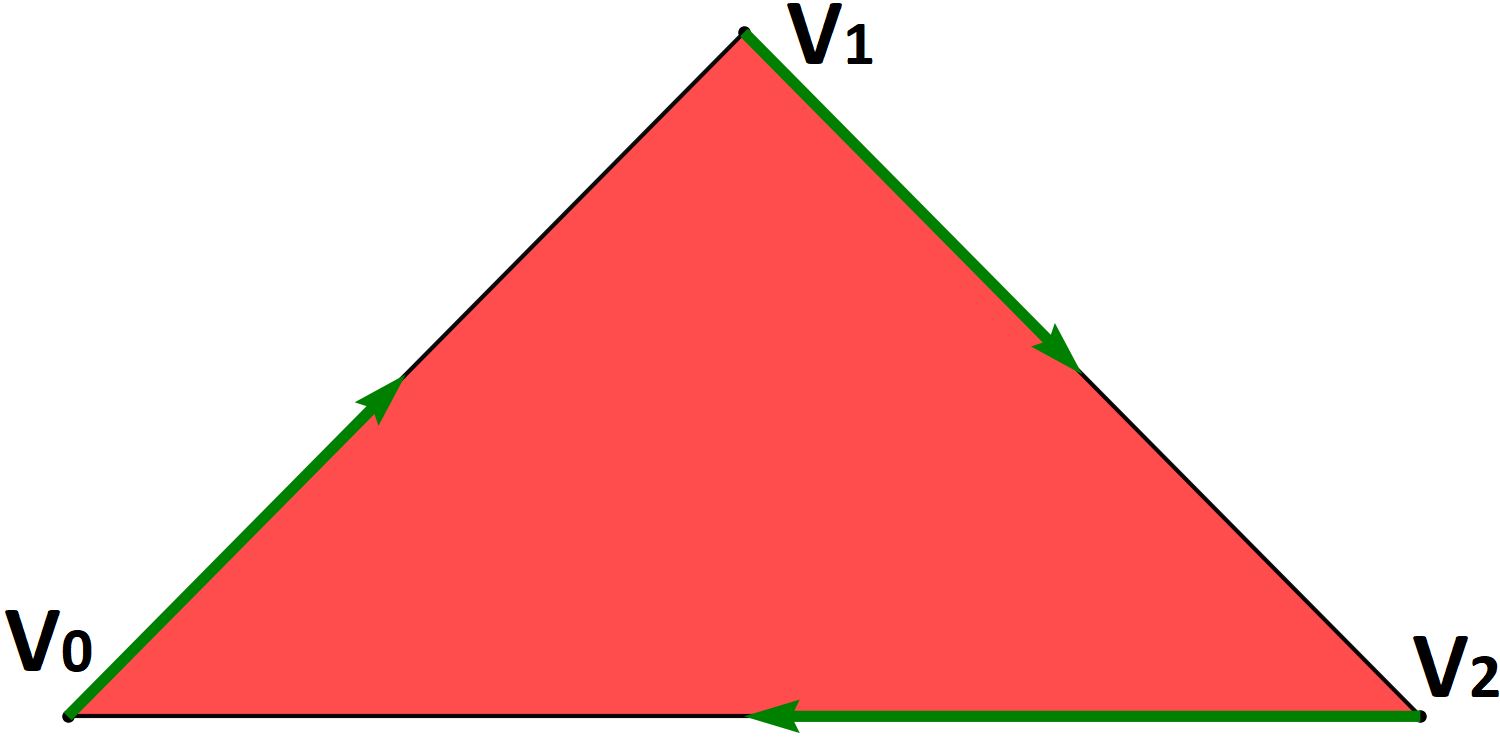}
  \caption{The combinatorial dynamical system induce by the matching matrix $ A $ from (\ref{eqMatToyEx}).}
  \label{imgGrpNonOri}
 \end{figure}
			
		\subsection{Data Generated from Classical Dynamical Systems}
			\subsubsection{Lotka-Volterra}
				Consider the data points $ X = \{ [10i, 10j ] \mid i = 0, 1, 2, \ldots,7, 8, j = 0, 1, 2 \ldots,7, 8 \} $. We compute $ \dot{x} $ with the Lotka-Volterra equations: 				
				
				\begin{equation}\label{eqLoVolt}
					\begin{cases}
						\frac{dx}{dt} = (0.4 - 0.01y)x\\
						\frac{dy}{dt} = (0.005x -0.3)y
					\end{cases}.
				\end{equation}
				
				The dynamical system (\ref{eqLoVolt}) has an equilibrium point at $(40, 60)$ and $(0,0)$. Also, there is an infinity of cycles. We construct the Delaunay complex from the data and compute the optimal solution of the reduced minimization problem (\ref{eqPrRedOpt}) with $ \alpha = 0.95 $. In this problem, we have $ 642 $ simplices and $ 1881 $ admissible matchings. The optimal solution $A$ induce the combinatorial dynamical system at Figure \ref{imgLotkVolt}.
				
				\begin{figure}
  	 				\center
  					\includegraphics[height=5.5in, width=5.5in, scale=1.00, angle=0 ]{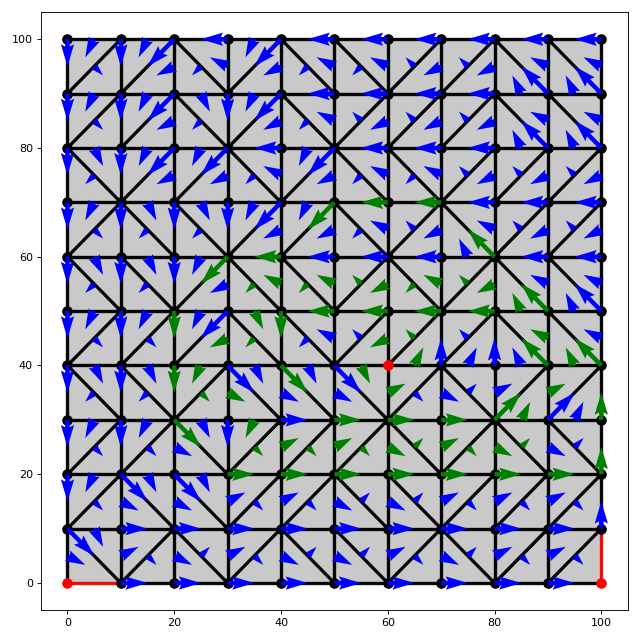}
  \caption{The combinatorial dynamical system obtained from the data with the Lotka-Volterra equations (\ref{eqLoVolt}).}
  					\label{imgLotkVolt}
 			\end{figure}
				We obtain three critical $0$-simplices and two critical $1$-simplices. Critical $1$-simplex has a similar dynamic of a saddle point but they are no saddle point in the classical dynamical system of the Lotka-Volterra. But they are needed to satisfy the definition of combinatorial dynamical system.
			
				For $i$-cycles, we have two $0$-cycles and two $1$-cycles. We find some recurrent behaviour as expected. We also have some errors on the boundary of the combinatorial dynamical system. We need to be careful when we analyze the data on the boundary of the simplicial complex. Some critical simplices are caused by finite data. As an example, the bottom right critical $0$-simplex in Figure \ref{imgLotkVolt} should not be there. There is no stationary fixed point in the classical dynamical system of (\ref{eqLoVolt}). 
			
			\subsubsection{Lorenz Attractor}
			 	We take a linear approximation of the trajectory at the initial values at $ x_0 = (0.00, 1.00, 1.05) $.  Let $ x_{i+1} = x_i + \Delta t \; \dot{x_i}  $ with $ \Delta t = 0.2 $ and $ i = 0, 1, 2 \ldots 999 $. We obtain $ 1000 $ data points. We compute $ \dot{x} $ with the Lorenz Attractor equations:
			 	
			 	\begin{equation}\label{eqLorAtt}
			 		\begin{cases}
			 			\frac{dx}{dt} = 10(y-x) \\
			 			\frac{dy}{dt} = 28x -xz - y\\
			 			\frac{dz}{dt} = xy-\frac{8}{3}z
			 		\end{cases}.
			 	\end{equation}
			 	
			We want to capture the dynamics of the chaotic attractor of (\ref{eqLorAtt}). We construct the Delaunay complex from data points. We obtained $8113$ $3$-simplices. We solve the problem (\ref{eqPrRedOpt}) to reduce the number of variables. We obtain $0$ critical $0$-simplices, $20$ critical $1$-simplices, $49$ critical $2$-simplices and $28$ critical $3$-simplices. We compute the strongly connected components of the directed graph from multivalued map $ \Pi_{\mathcal{V}}$ to find the recurrent behaviour. The recurrent behaviour is an union of $i$-cycle. We have $4$ $0$-cycles, $6$ $1$-cycles and $1$ $2$-cycle. At Figure (\ref{imgLorAtt}), we see the biggest strongly connected components. It's a $1$-cycle with $3528$ simplices and $421$ self-intersections. The dynamic of this attractor is very complex. This is caused by the chaotic behaviour of the Lorenz attractor.
			 	
				\begin{figure}
  	 				\center
  					\includegraphics[height=9cm, width=14cm, scale=1.00, angle=0 ]{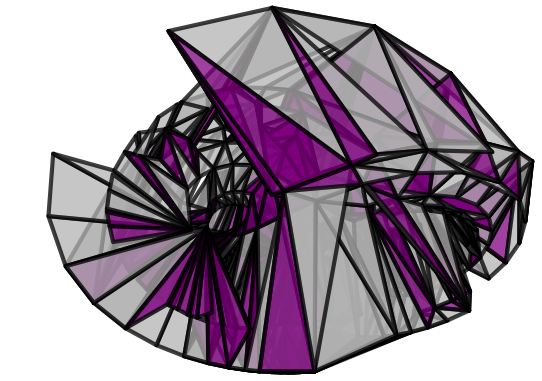}
  \caption{The biggest strongly connected components from the combinatorial dynamical system from the data of the Lorenz Attractor equations. This is a $1$-cycle. It has $3538$ simplices, including $421$ simplices which are auto-intersections coloured in purple.}
  					\label{imgLorAtt}
 			\end{figure}

	\section{Extensions to the Algorithm}	
	
		\subsection{Barycentric Subdivision}	
		
		We can apply a barycentric subdivision before the construction of the minimization problem (\ref{eqPrOp}). Let us start with a simple example to prove it is interesting to do that.
		
		Let $ x_0 $ be a $2$-simplex and $ x_1, x_2 $ and $x_3$ be $1$-simplices and faces of $ x_0 $. Suppose that $ c_{01} = c_{02} = c_{03}  < \alpha$. By the constraints of (\ref{eqPrOp}), the solution can only match one $1$-simplex with $ x_0 $. But $ x_1, x_2 $ and $ x_3 $ are good candidates for matching, because they have low costs. By applying a barycentric subdivision, we can match $ x_1, x_2 $ and $ x_3 $ with a $2$-simplex in the interior of $ x_0 $.
		
		Let us define the barycentric subdivision. Let be $K$ the simplicial complex and $K'$ be the resulted simplicial complex after applying the barycentric subdivision of $K$. We proceed by induction on the dimension of $K$. For each simplex $ \sigma \in K_n$, we add the barycentre of $\sigma$ as a new vertex in $ K' $ and connect the barycentre of $ \sigma $ with the other barycentre at the boundary of $ \sigma $ to construct the new $n$-simplices of $K'$. 
		
		Let $ v : K \to \mathbb{R} $ be the vector associated to a simplex and let us construct $ v' : K' \to \mathbb{R} $. For each $\sigma \in K_0 $, $ \sigma \in K_0' $. We set $ v'(\sigma) = v(\sigma) $. For the other simplex $ \sigma \in K' $, there exists a simplex $ \tau \in K $ such as $ \sigma $ is in the interior of $ \tau$, then put $ v'(\sigma) = v(\tau)$. 
		
		We try this approach on an example. Let $ x = \{ (-1, -1), (1, -1), (0, 2) \} $ and $ x' = \{ (1, 1), (-1, 1),$ $ (0, -2) \} $. We have a vector field sample from the dynamical system:
		
		\begin{equation*}
		\begin{cases}
			\frac{dx}{dt} = -x \\
			\frac{dy}{dt} = -y
		\end{cases}.
		\end{equation*}

 	We use the Delaunay complex to construct the simplicial complex. We should obtain a fixed point at $ (0,0) $. But there is no $0$-simplex at $ (0,0) $. The solution of the minimization will also put a $0$-simplex somewhere on the vertices. But this won't represent the finite vector field data correctly. If we apply a barycentric subdivision, then it will add a $0$-simplex at the center of the $ 2 $-simplex which is closed to the real fixed point.
		
		\begin{figure}
  			\center
  			\subfigure[ The matching from the solution obtained by solving the minimization problem without applying a barycentric subdivision. ]{
   				\includegraphics[height=5.715cm, width=7.62cm, scale=1.00, angle=0 ]{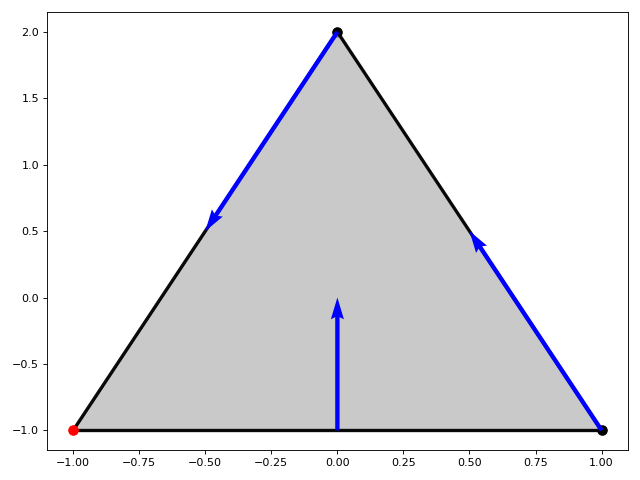}\label{subBefBaryCDS}
  			}
  			\,
  			\subfigure[ The matching from the solution obtained by solving the minimization problem with a barycentric subdivision applied to the data. ]{
   				\includegraphics[height=5.715cm, width=7.62cm, scale=1.00, angle=0 ]{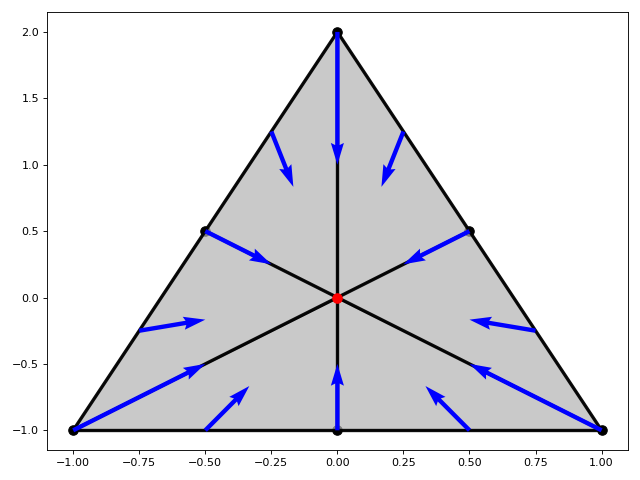}\label{subAftBaryCDS}
  }
  			\caption{An example on applying a barycentric subdivision before the matching algorithm. It gives a better result.}
  			\label{figsubBaryCDS}
 		\end{figure}
		
		In summary, applying a barycentric subdivision will help to add more details to the combinatorial dynamical system. But there is an inconvenience to the barycentric subdivision. This will add more simplices to the simplicial complex and it will take more time to compute the solution.
		
	\subsection{Gradient Combinatorial Dynamical System}
		In this subsection, we demonstrate two different methods to obtain a gradient combinatorial dynamical system with a similar approach as before. The first method is to choose an $ \alpha $ small enough that the solution $A$ from the minimization (\ref{eqPrOp}) induces a matching of a gradient combinatorial dynamical system. The second method is to add constraints to the minimization problem (\ref{eqPrOp}), so to remove all cycles in the solution. We say that a combinatorial dynamical system is gradient if there is no elementary cycle with length greater than $1$. In this case, the image of an elementary cycle has only $1$ critical simplex.
		
		For the first method, if we choose $ \alpha = 0 $, then the minimum of the problem (\ref{eqPrOp}) is obtained by setting $ a_{ii} = 1 $ for all $ i = 0, 1, \ldots, N-1 $. This set all simplices to critical and it is a gradient combinatorial dynamical system.
		
		\begin{lemma}
			Let $ c_{ij} > 0 $ for all $ i \neq j $ and $ l = \min_{i \neq j}\{ c_{ij} \} $. If $ \alpha < \frac{l}{2} $, then the global minimum of  (\ref{eqPrOp}) is $ a_{ii} = 1 $ for all $ i = 0, 1, \ldots, N-1 $ and $ a_{ij} = 0 $ for $ i \neq j $.
		\end{lemma}
		\begin{proof}
			We argue by contradiction. Let $ A $ be a global minimum of (\ref{eqPrOp}) and suppose that there exist $i \neq j $ such that $ a_{ij} = 1 $. Consider $ c_{ii} + c_{jj} = 2 \alpha < l \leq c_{ij} $. We construct $B$. Set $ b_{km} = a_{km}$ for $k = 0, 1,2, \ldots, N-1 $ and $ m = 0,1,2, \ldots N-1 $, except for $ a_{ij} = 0 $, $ a_{ii} = 1 $ and $ a_{jj} = 1 $. We have $ f(B)  < f(A) $. We get a contradiction and the global minimum is $a_{ii} = 1 $ for all $ i = 0, 1, \ldots, N-1 $ and $ a_{ij} = 0 $ for $ i \neq j $.
		\end{proof}
		
		The main idea for this method is to choose the greatest $ \alpha \in [0, 2] $ such that the solution obtained from (\ref{eqPrOp}) is induced a gradient combinatorial vector field. The advantage of this approach is that it is fast to compute. But it gives more critical simplices than the second method. 
		
		For the second method, we add a new set of constraints to the integer linear problem (\ref{eqPrOp}), so to make no cycle in the solution. The new minimization problem is:
		
		\begin{equation}\label{eqPrOpGradient}
			\begin{aligned}
			& \underset{A \in M_N(\{0,1\}) }{\text{minimize}}
			& & f(A) := \sum_{i=0}^{N-1}\sum_{j=0}^{N-1} a_{ij} c_{ij} \\
			& \text{sujet à}
			& &  \sum_{i=0}^{N-1} a_{ik} + \sum_{j=0}^{N-1} a_{kj} - a_{kk}  = 1 , \quad k =0, 1, \ldots, N-1  \\
			& & & \sum_{a_ij \in c} a_{ij} < | c |, \quad c \in C
			\end{aligned}
	\end{equation}	
	where $ C $ is the set of all possible elementary cycles with length greater than $1$ on $K$. Let $K$ a simplicial complex and $d$ the dimension of $K$. The elementary cycles can be decomposed in $i$-cycle:
	\begin{equation*}
		C = \bigcup_{i = 0}^{d-1} \{ \text{all possible elementary } i\text{-cycles} \}.
	\end{equation*}
		
	\begin{theorem}\label{thGloMinOptGrad}
		If $ A $ is a global minimum to the ILP (\ref{eqPrOpGradient}), then $ A $ induce a matching that satisfies the condition of a gradient combinatorial dynamical system.
	\end{theorem}
	The proof is similar that to the Theorem \ref{thGloMinOpt} but with the new set of constraints the solution $ A $ of (\ref{eqPrOpGradient}) induces a matching with no cycle. If we set $ \alpha $ high, this approach has less critical simplices than the first method. But we have a lot of more constraints to compute.

	We end with an example comparing two methods. Let $ X = \{ (0,0), (1,1), (2,0) \} $ and $ \dot{X} = \{ (0.05, 1), (1,0), (-1,-1) \} $. With this dataset, if we do not have a low $\alpha$ the solution $A$ induces a cycle in the combinatorial dynamical system. We choose $ \alpha = 0.14 $ to obtain an admissible matching with no-cycle. For the second method, we need to add two constraints to remove all possible cycle. The comparison of methods is displayed in Figure \ref{figGradCDS}.
	
	\begin{figure}
  		\center
  		\subfigure[ The solution from the minimization problem with $ \alpha = 0.15 $. The combinatorial dynamical system has a $1$-cycle. ]{
   				\includegraphics[height=5.715cm, width=7.62cm, scale=1.00, angle=0 ]{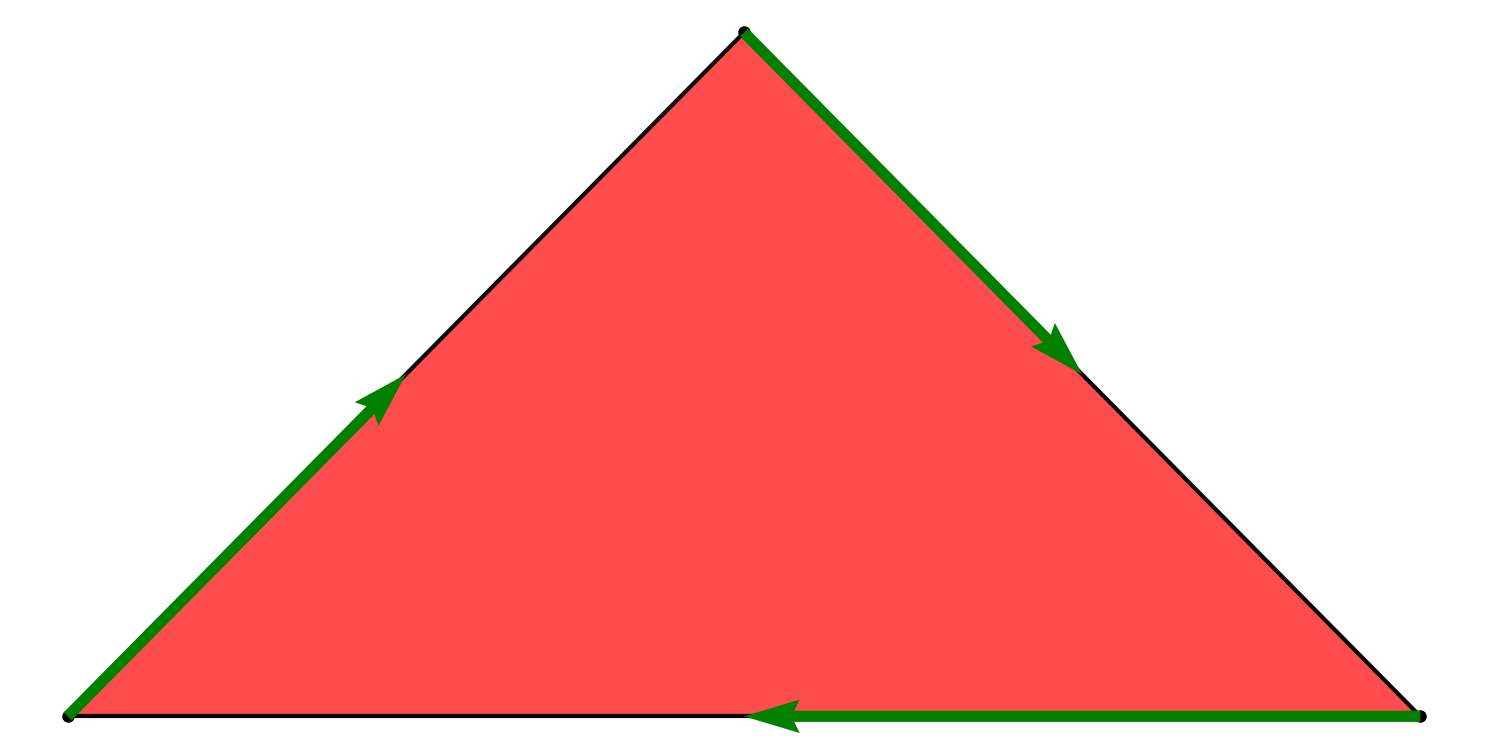}\label{subGradCDS1}
  			}
  			\,
  			\subfigure[ The solution from the minimization problem with $ \alpha = 0.14 $. The combinatorial dynamical system is gradient. ]{
   				\includegraphics[height=5.715cm, width=7.62cm, scale=1.00, angle=0 ]{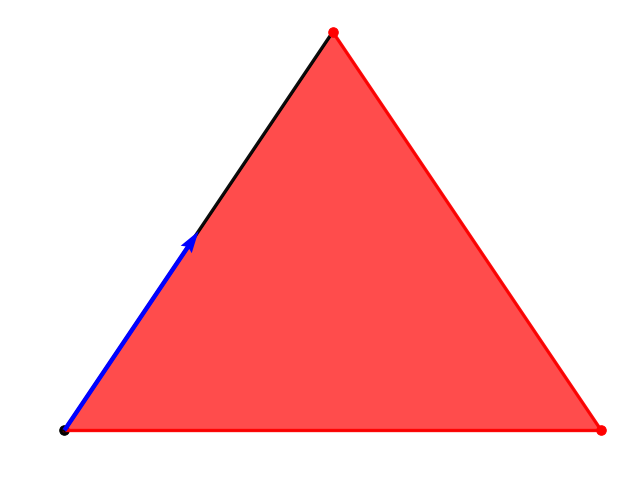}\label{subGradCDS2}
  }
  			\qquad
  			\subfigure[ The solution from the minimization problem with two constraints added to remove all cyclic solutions. The combinatorial dynamical system is gradient. ]{
   			\includegraphics[height=5.715cm, width=7.62cm, scale=1.00, angle=0 ]{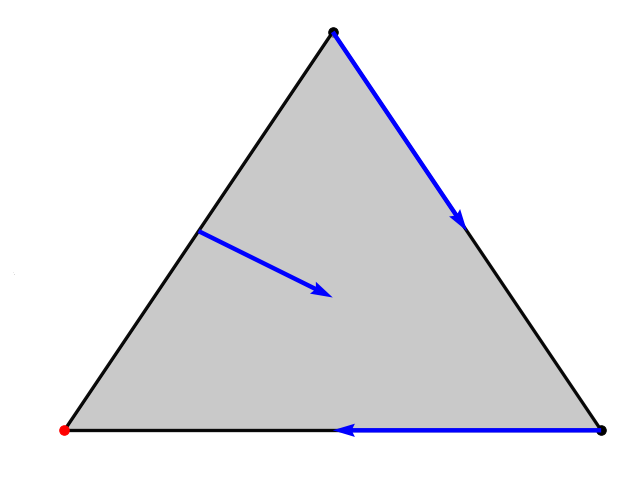}\label{subFigGradCDS3}
  }

  			\caption{Comparison of the different methods to make sure that the combinatorial dynamical system is gradient.} 
  			 \label{figGradCDS}
 		\end{figure}


	\clearpage
	\typeout{}
	\bibliography{biblio}

\end{document}